\documentclass[12 pt, reqno]{amsart}
\usepackage{amsmath, amsthm, amscd, amsfonts, amssymb, graphicx, color}
\usepackage[bookmarksnumbered, colorlinks, plainpages]{hyperref}

\newtheorem{theorem}{Theorem}[section]
\newtheorem{lemma}[theorem]{Lemma}

\theoremstyle{definition}

\numberwithin{equation}{section}

\begin{document}
\setcounter{page}{1}

\title[Rings for which f.g. projective modules have FI-extending Property]{Rings for which f.g. projective modules have \\ the FI-extending Property}

\author[P. Danchev]{P. Danchev}
\author[M. Zahiri]{M. Zahiri} 
\author[S. Zahiri]{S. Zahiri}

\address{Peter Danchev, Institute of Mathematics and Informatics, Bulgarian Academy of Sciences, 1113 Sofia, Bulgaria}
\email{\textcolor[rgb]{0.00,0.00,0.84}{pvdanchev@yahoo.com; danchev@math.bas.bg}}
\address{Masoome Zahiri, Department of  Mathematics, Faculty of Sciences, Higher Education center of Eghlid, Eghlid, Iran}
\email{\textcolor[rgb]{0.00,0.00,0.84}{m.zahiri86@gmail.com; m.zahiri@eghlid.ac.ir}}
\address{Saeide Zahiri, Department of  Mathematics, Faculty of Sciences, Higher Education center of Eghlid, Eghlid, Iran}
\email{\textcolor[rgb]{0.00,0.00,0.84}{saeede.zahiri@yahoo.com; s.zahiri@eghlid.ac.ir}}

\thanks{*Corresponding author}

\begin{abstract}
A right $R$-module $M$ is said to be {\it FI-extending} if any fully invariant submodule of $M$ is essential in a direct summand of $M$. In this short note we prove that if $R$ has ACC on the right annihilators, then $R_R$ is FI-extending if, and only if, every f.g. projective module is too FI-extending. This is an affirmative answer to the question raised by Birkenmeier-Park-Rizvi in Commun. Algebra on 2002 (see \cite{2}).
\end{abstract}
\maketitle

\vskip1.0pc

\noindent{\centerline {\bf Dedicated to the memory of Prof. Syed Mohammed Tariq Rizvi}}

\section{\textbf{Introduction and Motivation}}

Throughout this paper, $R$ denotes a (possibly noncommutative) ring with identity, and $M_R$ stands for a unital right $R$-module. In last decades the
theory of extending modules unambiguously play an important role in the theory of rings and modules. 

According to \cite{birken}, a module $M$ is called {\it FI-extending} if every fully invariant submodule of $M$ is essential in a direct summand of $M$. Note that the set of fully invariant submodules of a module $M$ includes the socle, the Jacobson radical, the torsion submodule for a torsion theory (e.g., the singular submodule $\mathcal{Z}(MR)$), and also $MI$ for all right ideals $I$ of $R$, etc. \\ 

Likewise, the class of the FI-extending modules is closed under taking direct sums, but it is still unknown whether or not a direct summand of an FI-extending module is again FI-extending. However, the authors in \cite{tas} proved that every direct summand of an FI-extending module having the (C3) property is again FI-extending. They also showed that any direct summand of an FI-extending module with the SIP-extending property is  FI-extending as well.\\

Moreover, the authors in \cite{2} studied the FI-extending modules and stated the following interesting question: characterize the classes of rings satisfying the condition that every (cyclic, finitely generated, projective, etc.) module is FI-extending.\\ 

In this brief article, we approach to the posed question and establish that any finitely generated projective module over an FI-extending ring with A.C.C on annihilators is also an FI-extending module.

\section{\textbf{The Main Result}}

Before formulating the chief result, we need a series of preliminaries.

\begin{lemma}\label{1} (\cite[(7.15) Theorem]{Lam2}) Let $R$ have ACC on right annihilators. Then, $\mathcal{Z}_r(R)$ is a nilpotent ideal.
\end{lemma}

\begin{lemma}\label{2} Let $R$ be a ring. Then, $\mathcal{Z}_r(M_n(R))=M_n(\mathcal{Z}_r(R))$.
\end{lemma}

\begin{lemma}\label{es} Let $R$ be a ring and $t$ a positive integer. Then, $K\leq_{es} R^t_R$ if, and only if, there exists $L\leq_{es} R_R$ such that $L^t\leq K$.
\end{lemma}

\begin{lemma}\label{3} Let $M$ be a module and $N$ a fully invariant submodule of $M$. If there exists an idempotent $E\in End_R(M)$ such that $N\leq_{es} E(M)$, then, for any $s\in End_R(M)$, there exists an essential submodule $K$ of $M$ such that $(1-E)sE(K)=(0)$.
\end{lemma}

\begin{proof} As $N\leq_{es} E(M)$, we must have $$N\oplus (1-E)(M)\leq_{es}M.$$ Put $K:=N\oplus (1-E)(M)$. Thus, we obtain $(1-E)sE(K)=(0)$ for any $s\in  End_R(M)$, as required.
\end{proof}

We are now managing to prove our principal result which gives a positive answer to the aforementioned problem from \cite{2}. 

\begin{theorem} Let $R$ have ACC on right annihilators. Then, $R$ is a right FI-extending if, and only if, any f.g. projective right module is FI-extending.
\end{theorem}

\begin{proof} Assume first that $P$ is a f.g. projective module and $P_1$ a fully invariant submodule of $P$. Since $P$ is a f.g. projective, there is a positive integer $t$ such that $P$ is a direct summand of $R^t$. So, $P\oplus Q=R^t$ for some $Q\leq R^t$ as right $R$-modules. Set $Q_1:=\sum Im(f)$, where $f\in Hom_R(P,Q)$. Therefore, one readily checks that $P_1\oplus Q_1$ is a fully invariant submodule of $P\oplus Q$. Since $R_R$ is FI-extending and, as noticed above, the FI-extending property is closed under direct sums, we can deduce that $R^t_R$ is FI-extending too. So, $P\oplus Q$ is FI-extending as well, and hence there exists an idempotent $$E=\left(
                               \begin{array}{cc}
                                 e &m  \\
                                 n & f \\
                               \end{array}
                             \right)\in \left(
                               \begin{array}{cc}
                                 End_R(M_1) & Hom_R(M_2,M_1) \\
                                 Hom_R(M_1,M_2) & End_R(M_2) \\
                               \end{array}
                             \right)$$ such that $(P_1\oplus Q_1)\leq_{es} E(P\oplus Q)$. But, since $E=E^2$, it must be that $$e=e^2+mn,\ m=em+mf,\ n=fn+ne,\ f=f^2+nm.$$ It now follows that $$em=m(1-f),\ (1-e)m=mf,\ ne=(1-f)n,\ fn-n(1-e).$$ \par
                            Furthermore, since $$\left(
                               \begin{array}{cc}
                                 mn &(1-e)m  \\
                                 -ne & -nm \\
                               \end{array}
                             \right)=\left(
                               \begin{array}{cc}
                                 1-e &-m  \\
                                 -n & 1-f \\
                               \end{array}
                             \right)\left(
                               \begin{array}{cc}
                                 1 &0  \\
                                 0 & 0 \\
                               \end{array}
                             \right)\left(
                               \begin{array}{cc}
                                 e &m  \\
                                 n & f \\
                               \end{array}
                             \right),$$ thanks to Lemma \ref{3} there is an essential submodule, say $P_2\oplus Q_2$, of $P\oplus Q$ such that $\left(
                               \begin{array}{cc}
                                 mn &(1-e)m  \\
                                 -ne & -nm \\
                               \end{array}
                             \right)(P_2,Q_2)=0$. As $R^t=P\oplus Q$, Lemma \ref{es} applies to get $L\leq_{es} R_R$ such that $L^t\leq (P_2\oplus Q_2)$. Consequently, $$\left(
                               \begin{array}{cc}
                                 mn &(1-e)m  \\
                                 -ne & -nm \\
                               \end{array}
                             \right)(L^t)=0.$$ It, thereby, follows that $$\left(
                               \begin{array}{cc}
                                 mn &(1-e)m  \\
                                 -ne & -nm \\
                               \end{array}
                             \right)M_t(L)=0.\indent (*)$$ Obviously, $M_t(L)\leq_{es} M_t(R)$ as $M_t(R)$-right module. Since $R^t=P\oplus Q$, we can write $$M_t(R)=End_R(R^t)=End_R(P\oplus Q).$$ So, $\left(
                               \begin{array}{cc}
                                 mn &(1-e)m  \\
                                 -ne & -nm \\
                               \end{array}
                             \right)\in M_t(R)$. From this and $(*)$, along with Lemma~\ref{2} at hand, we detect that $$\left(
                               \begin{array}{cc}
                                 mn &(1-e)m  \\
                                 -ne & -nm \\
                               \end{array}
                             \right)\in \mathcal{Z}_r(M_t(R))=M_t(\mathcal{Z}_r(R)).$$ As Lemma~\ref{1} tells us that $\mathcal{Z}_r(R)$ is a nilpotent ideal, we conclude $mn$ is nilpotent. Hence, there exists a positive integer $s$ such that $(mn)^k=(nm)^k=0$.\\
          We next claim that $$(P_1\oplus Q_1)\leq_{es} \left(
                               \begin{array}{cc}
                                 e^k &0  \\
                                 0 & f^k \\
                               \end{array}
                             \right)E\left(
                               \begin{array}{cc}
                                 e^k &0  \\
                                 0 & f^k \\
                               \end{array}
                             \right)(P\oplus Q).$$ In fact, to demonstrate that, it is easy to verify that 
                             $$\left(
                               \begin{array}{cc}
                                 e^k &0  \\
                                 0 & f^k \\
                               \end{array}
                             \right)E\left(
                               \begin{array}{cc}
                                 e^k &0  \\
                                 0 & f^k \\
                               \end{array}
                             \right)(P_1\oplus Q_1)=P_1\oplus Q_1,$$ which forces that $$(P_1\oplus Q_1)\leq \left(
                               \begin{array}{cc}
                                 e^k &0  \\
                                 0 & f^k \\
                               \end{array}
                             \right)E\left(
                               \begin{array}{cc}
                                 e^k &0  \\
                                 0 & f^k \\
                               \end{array}
                             \right)(P\oplus Q).$$ 
                             
                             Now, assume that $$\left(
                               \begin{array}{cc}
                                 e^k &0  \\
                                 0 & f^k \\
                               \end{array}
                             \right)E\left(
                               \begin{array}{cc}
                                 e^k &0  \\
                                 0 & f^k \\
                               \end{array}
                             \right)(p, q)\neq 0\ \text{for some}\ (p, q)\in (P\oplus Q).$$ It, thus, follows that $ E\left(
                               \begin{array}{cc}
                                 e^k &0  \\
                                 0 & f^k \\
                               \end{array}
                             \right)(p, q)\neq 0$. Since $E\left(
                               \begin{array}{cc}
                                 e^k &0  \\
                                 0 & f^k \\
                               \end{array}
                             \right)(p, q)$ is a non-zero element of $E(P\oplus Q)$ and $(P_1\oplus Q_1)\leq_{es} E(P\oplus Q)$, there exists an element $r\in R$ such that $0\neq E\left(
                               \begin{array}{cc}
                                 e^k &0  \\
                                 0 & f^k \\
                               \end{array}
                             \right)(pr, qr)\in (P_1\oplus Q_1)$. Besides, since $\left(
                               \begin{array}{cc}
                                 e^k &0  \\
                                 0 & f^k \\
                               \end{array}
                             \right)(p_1,q_1)=(p_1,q_1)$ for every $(p_1,q_1)\in (P_1\oplus Q_1)$, we may write
                              $$0\neq \left(
                               \begin{array}{cc}
                                 e^k &0  \\
                                 0 & f^k \\
                               \end{array}
                             \right)E\left(
                               \begin{array}{cc}
                                 e^k &0  \\
                                 0 & f^k \\
                               \end{array}
                             \right)(pr, qr)=E\left(
                               \begin{array}{cc}
                                 e^k &0  \\
                                 0 & f^k \\
                               \end{array}
                             \right)(pr, qr)\in (P_1,Q_1),$$ which yields that
                             $$(P_1\oplus Q_1)\leq_{es} \left(
                               \begin{array}{cc}
                                 e^k &0  \\
                                 0 & f^k \\
                               \end{array}
                             \right)E\left(
                               \begin{array}{cc}
                                 e^k &0  \\
                                 0 & f^k \\
                               \end{array}
                             \right)(P\oplus Q).$$

                             On the other hand, one finds that $$\left(
                               \begin{array}{cc}
                                 e^k &0  \\
                                 0 & f^k \\
                               \end{array}
                             \right)E\left(
                               \begin{array}{cc}
                                 e^k &0  \\
                                 0 & f^k \\
                               \end{array}
                             \right)=\left(
                               \begin{array}{cc}
                                 e^k &0  \\
                                 0 & f^k \\
                               \end{array}
                             \right)\left(
                               \begin{array}{cc}
                                 e^{k+1} &mf^{k}  \\
                                 ne^{k} & f^{k+1} \\
                               \end{array}
                             \right)=$$$$\left(
                               \begin{array}{cc}
                                 e^{2k+1} &e^kmf^{k}  \\
                                 f^kne^{k} & f^{2k+1}\\
                               \end{array}
                             \right)=\left(
                               \begin{array}{cc}
                                 e^{2k+1} &m(1-f)^kf^{k}  \\
                                 n(1-e)^ke^{k} & f^{2k+1}\\
                               \end{array}
                             \right)=\left(
                               \begin{array}{cc}
                                 e^{2k+1} &0  \\
                                0 & f^{2k+1}\\
                               \end{array}
                             \right).$$ Thus, $$(P_1\oplus Q_1)\leq_{es} \left(
                               \begin{array}{cc}
                                 e^{2k+1} &0  \\
                                 0 & f^{2k+1} \\
                               \end{array}
                             \right)(P\oplus Q).\indent (**)$$ So, $(P_1\oplus Q_1)\leq_{es}\left(
                               \begin{array}{cc}
                                 e^{2k+1} &0  \\
                                0 & f^{2k+1}\\
                               \end{array}
                             \right)(P\oplus Q)$. It, therefore, follows that $$\left(
                               \begin{array}{cc}
                                 e^{2k+1} &0  \\
                                0 & f^{2k+1}\\
                               \end{array}
                             \right)(P\oplus Q)\cap (1-E)(P\oplus Q)=(0),$$ as claimed.\\ 
                             
                             We now assert that $$\left(
                               \begin{array}{cc}
                                 e^{2k+1} &0  \\
                                0 & f^{2k+1}\\
                               \end{array}
                             \right)S\oplus (1-E)S=S,$$ where $S=End_R(P\oplus Q)$.\\ First, to substantiate that, we show that $$\left(
                               \begin{array}{cc}
                                 e^{2k+1} &0  \\
                                0 & f^{2k+1}\\
                               \end{array}
                             \right)S\cap (1-E)S=(0),$$ as for otherwise there would exist $$0\neq s(P\oplus Q)\leq \left(
                               \begin{array}{cc}
                                 e^{2k+1} &0  \\
                                0 & f^{2k+1}\\
                               \end{array}
                             \right)(P\oplus Q).$$ But, as $$(P_1\oplus Q_1)\leq_{es}\left(
                               \begin{array}{cc}
                                 e^{2k+1} &0  \\
                                0 & f^{2k+1}\\
                               \end{array}
                             \right)(P\oplus Q),$$ we have $s(P\oplus Q)\cap (P_1\oplus Q_1)\neq (0)$. Since $(P_1\oplus Q_1)\leq_{es}E(P\oplus Q)$, we receive $s(P\oplus Q)\cap E(P\oplus Q)\neq 0$, arriving at a contradiction with $s\in (1-E)S$. Hence, $$\left(
                               \begin{array}{cc}
                                 e^{2k+1} &0  \\
                                0 & f^{2k+1}\\
                               \end{array}
                             \right)S\cap (1-E)S=(0).$$\\
From the equality $$(1-E)\left(
                               \begin{array}{cc}
                                 e^{2k} &0  \\
                                0 & 0\\
                               \end{array}
                             \right) =\left(
                               \begin{array}{cc}
                                 e^{2k}-e^{2k+1} &0  \\
                                -ne^{2k} & 0\\
                               \end{array}
                             \right),$$ we get that $$\left(
                               \begin{array}{cc}
                                 e^{2k} &0  \\
                                -ne^{2k} & 0\\
                               \end{array}
                             \right)     \in    \left(
                               \begin{array}{cc}
                                 e^{2k+1} &0  \\
                                0 & f^{2k+1}\\
                               \end{array}
                             \right)S\oplus (1-E)S.$$
              So, we may write that $$\left(
                               \begin{array}{cc}
                                 e^{2k} &0  \\
                                -ne^{2k} & 0\\
                               \end{array}
                             \right)   \left(
                               \begin{array}{cc}
                           0 &m  \\
                               0 & 0\\
                               \end{array}
                             \right)=\left(
                               \begin{array}{cc}
                                0& e^{2k}m   \\
                               0& -ne^{2k}m \\
                               \end{array}
                             \right)$$$$=\left(
                               \begin{array}{cc}
                                0& e^{2k}m   \\
                               0& -nm(1-f)^{2k} \\
                               \end{array}
                             \right)  \in    \left(
                               \begin{array}{cc}
                                 e^{2k+1} &0  \\
                                0 & f^{2k+1}\\
                               \end{array}
                             \right)S\oplus (1-E)S.\indent (*_1)$$

                             Furthermore, from $$(1-E)\left(
                               \begin{array}{cc}
                             0 &0  \\
                                0 & (1-f)^{2k}\\
                               \end{array}
                             \right) =\left(
                               \begin{array}{cc}
                                0 &-e^{2k}m   \\
                                0 &  (1-f)^{2k+1}\\
                               \end{array}
                             \right)$$ and $(*_1)$, we derive $$\left(
                               \begin{array}{cc}
                                0& 0 \\
                               0& -nm(1-f)^{2k}+(1-f)^{2k+1} \\
                               \end{array}
                             \right) =\left(
                               \begin{array}{cc}
                                0& 0 \\
                               0& -f(1-f)^{2k+1}+(1-f)^{2k+1} \\
                               \end{array}
                             \right)$$$$=\left(
                               \begin{array}{cc}
                                0& 0 \\
                               0& (1-f)^{2k+2} \\
                               \end{array}
                             \right) \in    \left(
                               \begin{array}{cc}
                                 e^{2k+1} &0  \\
                                0 & f^{2k+1}\\
                               \end{array}
                             \right)S\oplus (1-E)S.$$ Now, set $(1-f)^{2k+2} :=1+\Delta$, where $$\Delta=-(2k+2)f+(k+1)(2k+1)f^2-\frac{(k+1)(2k+1)2k}{3}f^3+\cdots+f^{2k+2}.$$ Then, it must be that $$\left(
                               \begin{array}{cc}
                                0& 0 \\
                               0& 1+\Delta \\
                               \end{array}
                             \right) \in    \left(
                               \begin{array}{cc}
                                 e^{2k+1} &0  \\
                                0 & f^{2k+1}\\
                               \end{array}
                             \right)S\oplus (1-E)S.\indent (*_3)$$ As $$\left(
                               \begin{array}{cc}
                                 0&0  \\
                                0 & f^{2k+1}\\
                               \end{array}
                             \right)=\left(
                               \begin{array}{cc}
                                 e^{2k+1} &0  \\
                                0 & f^{2k+1}\\
                               \end{array}
                             \right)\left(
                               \begin{array}{cc}
                                 0 &0  \\
                                0 & 1\\
                               \end{array}
                             \right)\in \left(
                               \begin{array}{cc}
                                 e^{2k+1} &0  \\
                                0 & f^{2k+1}\\
                               \end{array}
                             \right)S,$$ we obtain $$\left(
                               \begin{array}{cc}
                                 0 &0  \\
                                0 & \Delta f^{2k}\\
                               \end{array}
                             \right)\in \left(
                               \begin{array}{cc}
                                 e^{2k+1} &0  \\
                                0 & f^{2k+1}\\
                               \end{array}
                             \right)S.$$ From this and $(*_3)$, we have $$\left(
                               \begin{array}{cc}
                                 0 &0  \\
                                0 &  f^{2k}\\
                               \end{array}
                             \right)\in \left(
                               \begin{array}{cc}
                                 e^{2k+1} &0  \\
                                0 & f^{2k+1}\\
                               \end{array}
                             \right)S\oplus (1-E)S.$$ It guarantees that $$\left(
                               \begin{array}{cc}
                                 0 &0  \\
                                0 & \Delta f^{2k-1}\\
                               \end{array}
                             \right)\in \left(
                               \begin{array}{cc}
                                 e^{2k+1} &0  \\
                                0 & f^{2k+1}\\
                               \end{array}
                             \right)S\oplus (1-E)S.$$ From this and $(*_3)$, we extract $$\left(
                               \begin{array}{cc}
                                 0 &0  \\
                                0 & f^{2k-1}\\
                               \end{array}
                             \right)\in \left(
                               \begin{array}{cc}
                                 e^{2k+1} &0  \\
                                0 & f^{2k+1}\\
                               \end{array}
                             \right)S\oplus (1-E)S.$$ Continuing in this way, we arrive at $$\left(
                               \begin{array}{cc}
                                 0 &0  \\
                                0 & 1\\
                               \end{array}
                             \right)\in \left(
                               \begin{array}{cc}
                                 e^{2k+1} &0  \\
                                0 & f^{2k+1}\\
                               \end{array}
                             \right)S\oplus (1-E)S.$$ By symmetry, one plainly see that $$\left(
                               \begin{array}{cc}
                                 1 &0  \\
                                0 & 0\\
                               \end{array}
                             \right)\in \left(
                               \begin{array}{cc}
                                 e^{2k+1} &0  \\
                                0 & f^{2k+1}\\
                               \end{array}
                             \right)S\oplus (1-E)S.$$ So, $S=\left(
                               \begin{array}{cc}
                                 e^{2k+1} &0  \\
                                0 & f^{2k+1}\\
                               \end{array}
                             \right)S\oplus (1-E)S$, as needed.\\ Therefore, $$S(P,0)=\left(
                               \begin{array}{cc}
                                 e^{2k+1} &0  \\
                                0 & f^{2k+1}\\
                               \end{array}
                             \right)(P,0)\oplus (1-E)(P,0),$$ that is, $$S(P,0)=(e^{2k+1}(P),0)\oplus  \left\{\left(
                               \begin{array}{cc}
                                 (1-e)p &0  \\
                                 -np & 0 \\
                               \end{array}
                             \right)\mid \ p\in P\right\}.$$ \\ Finally, $$(P,0)=S(P,0)\cap P=(e^{2k+1}(P),0)\oplus  \left(\left\{\left(
                               \begin{array}{cc}
                                 (1-e)p &0  \\
                                 -np & 0 \\
                               \end{array}
                             \right)\mid \ p\in P\right\}\cap P\right),$$ which means that $e^{2k+1}(P)$ is a direct summand of $P$.\\ 
                             
                             On the other side, because in view of $ (**)$ we have $$(P_1\oplus Q_1)\leq_{es} \left(
                               \begin{array}{cc}
                                 e^{2k+1} &0  \\
                                 0 & f^{2k+1} \\
                               \end{array}
                             \right)(P\oplus Q),$$ we can write that $P_1\leq_{es} e^{2k+1}(P)$, i.e., $P$ is a FI-extending module, as wanted.
                           \end{proof}

\bibliographystyle{amsplain}

\end{document}